\documentclass[12pt]{amsart}
\usepackage{amscd}
\usepackage{amsmath}
\usepackage{amssymb}
\usepackage{units}
\usepackage[all]{xy}

\def\opn#1#2{\def#1{\operatorname{#2}}} 
\opn\chara{char} \opn\length{\ell} \opn\pd{pd} \opn\rk{rk}
\opn\projdim{proj\,dim} \opn\injdim{inj\,dim} \opn\rank{rank}
\opn\depth{depth} \opn\sdepth{sdepth} \opn\hdepth{hdepth}
\opn\grade{grade} \opn\height{height} \opn\embdim{emb\,dim}
\opn\codim{codim}  \opn\min{min} \opn\max{max}

\opn\Tr{Tr} \opn\bigrank{big\,rank}
\opn\superheight{superheight}\opn\lcm{lcm}
\opn\trdeg{tr\,deg}
\opn\reg{reg} \opn\lreg{lreg} \opn\ini{in} \opn\lpd{lpd}
\opn\size{size}

\opn\Spec{Spec} \opn\Supp{Supp} \opn\supp{supp} \opn\Sing{Sing}
\opn\Ass{Ass} \opn\Min{Min}
\opn\Ann{Ann} \opn\Rad{Rad} \opn\Soc{Soc}
\opn\Im{Im} \opn\Ker{Ker} \opn\Coker{Coker} \opn\Am{Am}
\opn\Hom{Hom} \opn\Tor{Tor} \opn\Ext{Ext} \opn\End{End}
\opn\Aut{Aut} \opn\id{id}  \opn\deg{deg}

\opn\nat{nat}
\opn\pff{pf}
\opn\Pf{Pf} \opn\GL{GL} \opn\SL{SL} \opn\mod{mod} \opn\ord{ord}
\opn\Gin{Gin} \opn\Hilb{Hilb}
\let\iso=\cong

\let\to=\rightarrow

\def\Implies{\ifmmode\Longrightarrow \else
        \unskip${}\Longrightarrow{}$\ignorespaces\fi}
\def\implies{\ifmmode\Rightarrow \else
        \unskip${}\Rightarrow{}$\ignorespaces\fi}
\def\iff{\ifmmode\Longleftrightarrow \else
        \unskip${}\Longleftrightarrow{}$\ignorespaces\fi}

\let\:=\colon
\newtheorem{Theorem}{Theorem}
\newtheorem{Lemma}{Lemma}

\newtheorem{Example}{Example}

\newtheorem{Conjecture}{Conjecture}

\let\epsilon\varepsilon
\let\phi=\varphi
\let\kappa=\varkappa
\def\qed{\ifhmode\textqed\fi
      \ifmmode\ifinner\quad\qedsymbol\else\dispqed\fi\fi}
\def\textqed{\unskip\nobreak\penalty50
       \hskip2em\hbox{}\nobreak\hfil\qedsymbol
       \parfillskip=0pt \finalhyphendemerits=0}
\def\dispqed{\rlap{\qquad\qedsymbol}}

\opn\dis{dis}
\def\pnt{{\raise0.5mm\hbox{\large\bf.}}}

\opn\Lex{Lex}
\begin{document}

\title{Four generated, squarefree,  monomial  ideals}
\author{Adrian Popescu}
\address{Adrian Popescu,  Department of Mathematics, University of Kaiserslautern, Erwin-Schr\"odinger-Str., 67663 Kaiserslautern, Germany}
\email{popescu@mathematik.uni-kl.de}
\author{Dorin Popescu}
\address{Dorin Popescu,  Simion Stoilow Institute of Mathematics of Romanian Academy, Research unit 5, P.O.Box 1-764, Bucharest 014700, Romania}
\email{dorin.popescu@imar.ro}
\thanks{The  support of the first author from the Department of Mathematics of the University of Kaiserslautern and the support of the second author from  grant  PN-II-RU-TE-2012-3-0161  of Romanian Ministry of Education, Research and Innovation are gratefully acknowledged.}
%
%

\begin{abstract}
$I\supsetneq J$ be  two  squarefree monomial ideals of a polynomial algebra over a field generated in degree $\geq d$, resp. $\geq d+1$ .  Suppose that  $I$ is either generated  by three monomials of degrees $d$ and a set of monomials of degrees $\geq d+1$, or by four special monomials of degrees $d$. If  the Stanley depth of $I/J$ is $\leq d+1$ then the usual depth of $I/J$ is $\leq d+1$ too.

Monomial Ideals,  Depth, Stanley depth.

2010 Mathematics Subject Classification: Primary 13C15, Secondary 13F20, 13F55,
13P10.
\end{abstract}

\maketitle

\section*{Introduction}
Let $K$ be a field and $S=K[x_1,\ldots,x_n]$ be the polynomial $K$-algebra in $n$ variables. Let  $I\supsetneq J$ be  two   squarefree monomial ideals of $S$ and suppose that  $I$ is generated by squarefree monomials of degrees $\geq d$   for some positive integer $d$.  After  a multigraded isomorphism we may assume either that $J=0$, or $J$ is generated in degrees $\geq d+1$.
  By \cite[Proposition 3.1]{HVZ} (see  \cite[Lemma 1.1]{P}) we have $\depth_S I/J\geq d$.
  Depth of $I/J$ is a homological invariant and  depends on the characteristic of the field $K$.

  The purpose of our paper is to study upper bound conditions for $\depth_SI/J$. Let  $B$ (resp. $C$) be the set of the squarefree monomials of degrees $d+1$  (resp. $d+2$) of $I\setminus J$. Suppose that $I$ is generated by some squarefree monomials $f_1,\ldots,f_r$ of degrees $ d$  for some $d\in {\mathbb N}$ and a set of squarefree monomials $E$ of degree $\geq d+1$. If $d=1$ and each monomial of $B\setminus E$ is the least common multiple of two $f_i$ then it is easy to show that $\depth_SI/J=1$ (see Lemma \ref{1}). Trying to extend this result for $d>1$ we find an obstruction given by Example \ref{ex}. Our extension
given by Lemma \ref{e} is just  a  special form, but  a natural condition seems to be given in terms of the Stanley depth.

   More precisely, let $P_{I\setminus J}$  be the poset of all squarefree monomials of $I\setminus J$  with the order given by the divisibility. Let $ P$ be a partition of  $P_{I\setminus J}$ in intervals $[u,v]=\{w\in  P_{I\setminus J}: u|w, w|v\}$, let us say   $P_{I\setminus J}=\cup_i [u_i,v_i]$, the union being disjoint.
Define $\sdepth  P=\min_i\deg v_i$ and  the  {\em Stanley depth} of $I/J$ given by $\sdepth_SI/J=\max_{ P} \sdepth  P$, where $ P$ runs in the set of all partitions of $P_{I\setminus J}$ (see  \cite{HVZ}, \cite{S}).
 Stanley's Conjecture says that  $\sdepth_S I/J\geq \depth_S I/J$. The Stanley  depth of $I/J$ is a combinatorial invariant and  does not depend on the characteristic of the field $K$. Stanley's Conjecture holds when $J=0$ and $I$ is an intersection of four monomial prime ideals by \cite{AP}, \cite{Po}, or $I$ is such that the sum of every three different of its minimal prime ideals is a constant ideal by \cite{P0} (see also \cite{P2}), or $I$ is  an intersection of three monomial primary ideals by \cite{Z}, or a monomial almost complete intersection by \cite{Ci}.

 \begin{Theorem} (D. Popescu \cite[Theorem 4.3]{P})\label{P} If $\sdepth_SI/J=d$ then  $\depth_SI/J=d$, that is Stanley's Conjecture holds in this case.
\end{Theorem}

Next step in the study of Stanley's Conjecture is  to show the following weaker conjecture.

\begin{Conjecture} \label{c}   Suppose that $I \subset S$ is minimally generated by some squarefree monomials $f_1,\ldots,f_r$ of degrees $d$,  and a  set $E$  of squarefree monomials of degrees $\geq d+1$. If  $\sdepth_S I/J=d+1$ then $\depth_S I/J \leq d+1$.
\end{Conjecture}

Set $s=|B|$, $q=|C|$. In the study of the above conjecture very useful seem to be the following two particular results of \cite[Theorem 1.3]{P1} and  \cite[Theorem 2.4]{Sh}.

\begin{Theorem} (D. Popescu)\label{po} If $s>q+r$ then $\depth_SI/J\leq d+1$.
\end{Theorem}

\begin{Theorem} (Y. Shen) \label{sh} If $s<2r$ then  $\depth_SI/J\leq d+1$.
\end{Theorem}

 These results were hinted by Stanley's Conjecture since it is obvious that $s>q+r$, or $s<2r$ imply   $\sdepth_SI/J\leq d+1$. The proof of Theorem \ref{po} uses Koszul homology (see \cite[Section 1.6]{BH}). Shen's proof of the above theorem as well of Theorem \ref{po} is easy and uses  the Hilbert depth considered by Bruns-Krattenhaler-Uliczka \cite{BKU} (see also \cite{U}, \cite{IM}).

 An equivalent definition for the Stanley depth is: $$\sdepth (M) = \max\{\sdepth \mathcal D\ |\ \mathcal D \textnormal{ is a Stanley decomposition of } M\},$$ where a Stanley decomposition of a $\mathbb Z-$graded (resp. $\mathbb Z^n-$graded) $S-module$ $M$ is $\mathcal D = (S_i, u_i)_{i\in I}$, where $u_i$ are homogenous elements of $M$ and $S_i$ are graded (resp. $\mathbb Z^n-$graded) $K-$algebra retracts of $S$ and $S_i \cap \Ann(u_i) = 0$ such that $M = \oplus_i S_i u_i $; and $\sdepth \mathcal D$ is the $\depth$ of the $S-$module $\oplus_i S_i u_i$. A more general concept is the one of Hilbert depth of a $\mathbb Z-$graded module $M$, denoted by $\hdepth_1 (M)$. Instead of considering equality, we only assume that $M \iso \oplus S_i (-s_i)$, where $s_i \in \mathbb Z$. One can also construct $\hdepth_n$ analogously if $M$ is a multigraded (that is $\mathbb Z^n)$ module.

 In \cite{AP1} is presented (and implemented) an algorithm that computes $\hdepth_1 (M)$ based on a Theorem of Uliczka \cite{U}; and in \cite{IZ} was presented an algorithm that computes $\hdepth_n (M)$. Meanwhile, another algorithm that computes $\hdepth_1$ and more was given in \cite{BMU}. \cite[Proposition 1.9]{AP1} gives a partial answer to a question of Herzog asking whether $\sdepth m = \sdepth (S \oplus m)$, where $m$ is the graded maximal ideal of $S$. More precisely, for $n \in \{1,2,3,4,5,7,9,11\}$ one obtains $\hdepth_1 m = \hdepth_1 (S \oplus m)$, which gives $\sdepth m = \sdepth (S \oplus m)$ (again Hilbert depth helps the study of Stanley depth). For  $n=6$ we have $\hdepth_1 m \neq \hdepth_1 (S \oplus m)$, which  means that in general Herzog's question could have  a negative answer. Later Ichim and Zarojanu checked the case $n=6$ and found indeed a counterexample to Herzog's question, which will be included in the new version of \cite{IZ}.

An important step in proving Conjecture \ref{c} is the following theorem.
\begin{Theorem} (D. Popescu-A. Zarojanu \cite{PZ1}, \cite[Theorem 1.5]{PZ2}) \label{pz} Conjecture \ref{c} holds  in each of the following two  cases:
\begin{enumerate}
\item{} $r= 1$,
\item{} $1<r\leq 3$, $E=\emptyset$.
\end{enumerate}
\end{Theorem}

Next theorem is the main result of this paper.

\begin{Theorem}   Conjecture \ref{c} holds  in each of the following two  cases:
\begin{enumerate}
\item{} $r\leq 3$,
\item{} $r=4$, $E=\emptyset$ and there exists $c\in C$ such that $\supp c\not \subset \cup_{i\in [4]} \supp f_i$.
\end{enumerate}
\end{Theorem}
This follows from our Theorems \ref{m}, \ref{m1}. The proof of \ref{m} extends the proof of   \cite[Theorem 2.3]{PZ2}.

We owe thanks to A. Zarojanu, who noticed  some small mistakes in a previous version of this paper and gave us the bad example \ref{bad}.

\section{Depth and Stanley depth}

Let $I\supsetneq J$ be  two  squarefree monomial ideals of $S$. We assume that $I$ is generated by squarefree monomials $f_1,\ldots,f_r$ of degrees $ d$  for some $d\in {\mathbb N}$ and a set of squarefree monomials $E$ of degree $\geq d+1$. We may suppose that either $J=0$, or  is generated by some squarefree monomials of degrees $\geq d+1$.
 $B$ (resp. $C$) denotes the set of the squarefree monomials of degrees $d+1$  (resp. $d+2$) of $I\setminus J$.

\begin{Lemma} \label{ele}  Let $J\subset I$  be square free monomial ideals and $j\in [n]$ be such that $(J:x_j)\not =(I:x_j)$. Then $\depth_S(I:x_j)/(J:x_j)\geq\depth_S I/J$.
\end{Lemma}
\begin{proof} We have
$$\pd_S I/J\geq \pd_{S_{x_j}} (I/J)\otimes S_{x_j}= \pd_{S_{x_j}} ((I:x_j)/(J:x_j))\otimes S_{x_j}=
\pd_S ((I:x_j)/(J:x_j))$$
the last equality holds since $x_j$ does not appear among the generators of $(I:x_j)$ and $(J:x_j)$. Now it is enough to apply the Auslander-Buchsbaum Theorem.
\end{proof}

\begin{Lemma} \label{l1} Let $t\in [n]$. Suppose that $I\not =J+I\cap (x_t)$ and $\depth_SI/(J+I\cap (x_t))=d$. If $\depth_SI/J\geq d+1$ then $\depth_SI/J=d+1$.
\end{Lemma}
\begin{proof} In the following exact sequence
$$0\to (I:x_t)/(J:x_t)\xrightarrow{x_t} I/J\to I/(J+I\cap (x_t))\to 0$$
the first term has depth $d+1$ by the Depth Lemma. Now it is enough to apply the above lemma.
\end{proof}

 Let $w_{ij}$ be the least common multiple of $f_i$ and $f_j$ and set $W$ to be the set of all $w_{ij}\in B$.

 \begin{Lemma} \label{1} If $d=1$ and  $B\subset E\cup W$ then $\depth_SI/J=1$.
 \end{Lemma}
 \begin{proof} First suppose that  $E=\emptyset$, let us say $I=(x_1,\ldots,x_r)$.  Set $S'=K[x_1,\ldots,x_r]$, $I'=I\cap S'$, $J'=J\cap S'$. By hypothesis $B\subset S'$ and it follows that $(x_{r+1},\ldots,x_n)I\subset J$ and so $\depth_SI=\depth_{S'}I'=1$. But $\depth_{S}J\geq 2$, if $J\not =0$,  and  so $\depth_{S}I/J= 1$ by the Depth Lemma.

 Now, suppose that  $E\not =\emptyset$. In the following exact sequence
 $$0\to (x_1,\ldots,x_r)/J\cap   (x_1,\ldots,x_r)\to I/J\to I/(J,x_1,\ldots,x_r)\to 0$$
 the first term has depth 1 as above and the last term has depth $\geq d+1$ since it is generated by squarefree monomials of degrees $\geq 2$ from $E$. Again the Depth Lemma gives  $\depth_{S}I/J= 1$.
 \end{proof}

\begin{Lemma}\label{e}  Suppose that $I \subset S$ is  generated by some squarefree monomials $f_1,...f_r$ of degree $d$. Assume that for all $b\in B$ all divisors of $b$ of degree $d$ are among $\{f_1,\ldots,f_r\}$. Then $\depth_S I/J = d$.
\end{Lemma}
\begin{proof} Apply induction on $d\geq 1$. If $d=1$   then   apply the above lemma.
Assume $d>1$. We may suppose that $n \in \supp f_1$.  $(I:x_n)$ is an extension of a squarefree monomial ideal $I'$ of $S'=K[x_1,\ldots,x_{n-1}]$ which is generated in degree $\geq d-1$. Similarly $(J:x_n)$ is generated by a squarefree monomial ideal $J'$ of $S'$. Note that the generators of $I'$ of degree $d-1$ have the form $f'_i=f_i/x_n$ for  $f_i\in (x_n)$, and the squarefree monomials $B'$ of degrees $d$ from $I'\setminus J'$ have the form $b'=b/x_n$ for some $b\in (B\cap (x_n))$. Certainly we must consider also the case when $f_j\not\in (x_n)$. If $x_nf_j\in J$ then $f_j\in (J:x_n)$  is not in $B'$. Otherwise, $f_j=(x_nf_j)/x_n\in B'$. Note that   all divisors of degree $d-1$ of each $b'\in B'$ are among $f'_i$.
 By induction hypothesis we have $\depth_{S'} I'/J'= d-1$ and so $\depth_S(I:x_n)/(J:x_n)= d$. Now it is enough to apply Lemma \ref{ele}.
\end{proof}
An obstruction to improve Lemma \ref{1} and the above lemma is given by the following example.
\begin{Example}\label{ex3} {Let $n=5$, $d=2$, $r=5$, $I = (x_1x_2,x_1x_3,x_2x_3,x_1x_4,x_3x_5)$, \hfill\\
$J=(x_1x_2x_5,x_1x_4x_5,x_2x_3x_4,x_3x_4x_5)$, $B=\{x_1x_2x_3,x_1x_2x_4,x_1x_3x_4,x_1x_3x_5,x_2x_3x_5\}$.
We have $\depth_S I/J = 3$  because  $\depth_S S/J = 3$, $\depth_S S/I = 2$ and with the help of Depth Lemma. Note that each $b\in B$ is the least common multiple of two generators of $I$, but for example $b=x_1x_2x_4$ has $x_2x_4\not \in I$ as a  divisor of degree $2$.}
\end{Example}

 Let $C_2=C\cap W$ and $C_3$ be the set of all $c\in C$ having all divisors from $B\setminus E$ in $W$. In particular each monomial of $C_3 $ is the least common multiple of three of $f_i$. The converse is not true as shows the following example.

 \begin{Example} \label{ex} Let $n=4$, $d=2$, $r=3$, $f_1=x_1x_2$, $f_2=x_2x_3$, $f_3=x_3x_4$, $I=(f_1,f_2,f_3)$ and $J=0$. Then $c=x_1x_2x_3x_4$ is the least common multiple of $f_1,f_2,f_3$ but has a divisor $b=x_1x_2x_4\in B$ which is not the least common multiple of two $f_i$.
 \end{Example}

Next theorem is our key result, its proof is based on \cite[Theorem 2.1]{PZ2} and will be given in the last section. The main reason that this proof works for $r\leq 3$ but not for $r=4$ is that in the first case $|C_3|\leq 1$ but in the second one we may have $|C_3|=4$, which makes the things harder. However, for $r\geq 5$ will appear a new problem since we may have $B\subset W$ and $s\geq 2r$ (for example when $r=5$, $d=2$ we may have $s=10=2r$). We remind that by Theorem \ref{sh} we had to check Stanley's Conjecture only when $s\geq 2r$.

\begin{Theorem} \label{m} Conjecture \ref{c} holds for $r\leq 3$, the case $r=1$ being given in Theorem \ref{pz}.
\end{Theorem}

\begin{Example}\label{ex1} { Let $n=5$, $f_1=x_1x_2$, $f_2=x_1x_3$, $f_3=x_1x_4$,  $a=x_2x_3x_5$, $E=\{a\}$, $I=(f_1,f_2,f_3,a)$, $J=(x_4a)$.  We have $w_{12}=f_1x_3$, $w_{13}= f_1x_4$, $w_{23}=f_2x_4$. Set $c=w_{12}x_4$, $c_1=w_{12}x_5$,  $c_2=w_{23}x_5$, $c_3=w_{13}x_5$. Then $C=\{c,c_1,c_2,c_3\}$
 and  $B\setminus E=B\cap ( \cup_i [f_i,c_i])$. Thus $s=7$, $q=4$, $r=3$. It is easy to see that $\sdepth_SI/J= 3$. Indeed, note that $c_1$ is the only $c'\in C$ which is multiple of $a$. Suppose that there exists a partition $ P$ on $P_{I/J}$ with sdepth $4$. Then we have  necessarily in $ P$ the interval $[a,c_1]$.  If $ P$ contains the interval $[f_1,c]$ then it must contain also the intervals $[f_2,c_2]$ and so  $[f_3,c_3]$, but then $w_{13}\in [f_1,c]\cap [f_3,c_3]$, that is the union is not disjoint.   If $ P$ contains the interval $[f_1,c_3]$ then $ P$ contains either $[f_3,c]$, $[f_2,c_2]$, or $[f_2,c]$, $[f_3,c_2]$, in both cases the intersection of these two intervals contains $w_{23}$, which is false.  By Theorem \ref{m} we get  $\depth_SI/J\leq 3$, this inequality being in fact an equality.}
\end{Example}

\section{A special case of $r=4$}

\begin{Theorem} \label{m0} Suppose that $I \subset S$ is minimally generated by some  squarefree monomials $\{f_1,\ldots,f_r\}$ of degrees $d$ such that
there exists $c\in C$ with $\supp c\not \subset \cup_{i\in [r]} \supp f_i$.
If Conjecture \ref{c} holds for $r'<r$ and
 $\sdepth_SI/J=d+1$,
then $\depth_SI/J\leq d+1$.
\end{Theorem}
\begin{proof}  By \cite[Lemma 1.1]{PZ2} we may assume that $C\subset (W)$. By hypothesis,  choose $t\in \supp c$ such that  $t\not \in  \cup_{i\in [r]} \supp f_i$.  We may suppose that $B\cap (x_t)=\{x_tf_1,\ldots x_tf_e\}$ for some $1\leq e\leq r$.
 Set $ I_t=I\cap (x_t)$, $ J_t= J\cap (x_t)$ and $ U_t= I_t/ J_t$. Then $B_t$ generates $I_t$.

 First assume that $\sdepth_S U_t\leq  d + 1$.  It follows that $\depth_S U_t \leq d + 1$ by \cite[Theorem 4.3]{P}.  But
$U_t\cong (I : x_t)/(J : x_t)$ and so $\depth_S  U_t \geq \depth_S I/J$ by
Lemma \ref{ele}, which is enough.

 Now assume that
 $U_t $ has sdepth $\geq d+2$. Let  $P_{U_t}$  be a partition on $ U_t$ with sdepth $d + 2$ and let   $[b_i, c_i]$ be the disjoint intervals starting with $b_i=x_tf_i$, $i\in [e]$. We may suppose that $c_i\in C$ for $i\in [e]$.
 We have   $c_i=x_tw_{ik_i}$ for some $1\leq k_i\leq r$, $k_i\not = i$  because $C\subset (W)$. Note that $x_tf_{k_i}\in B$ and so  $k_i\leq e$. We consider
the intervals $[f_i,c_i]$. These intervals contain $x_tf_i$ and $w_{ik_i}$.  If $w_{ik_i} =w_{jk_j}$  for $i\not =j$ then we get $c_i=c_j$ which is false. Thus these intervals are disjoint.

   Let $I_e$ be the ideal generated by  $f_j$ for $ e<j\leq r$ and
  $B\setminus (\cup_{i=1}^ e[f_i, c_i])$. Set $J_e=I_e\cap J$ and $U_e=I_e/J_e$. Note that  $c_i\not \in I_e$ for any $i\in [ e]$.
 In the following exact sequence
$$0\to I_e/J_e\to I/J\to I/J+I_e\to 0$$
the last term has a partition of sdepth $d+2$ given by the intervals  $[f_i,c_i]$ for $1\leq i\leq e$. It follows that $I_e\not =J_e$ because $\sdepth_SI/J=d+1$.
 Then $\sdepth_SI_e/J_e\leq d+1$  using \cite[Lemma 2.2]{R} and so
$\depth_SI_e/J_e\leq d+1$ by Conjecture \ref{c} applied for $r'<r$. But the last term of the above sequence has depth $>d$ because
$x_t$ does not annihilate $f_i$ for $i\in [e]$. With the Depth Lemma we get $\depth_SI/J\leq d+1$.
\end{proof}

\begin{Example}\label{eex1}{ Let $n=5$, $r=4$, $f_1=x_2x_3$, $f_2=x_1x_2$, $f_3=x_3x_4$, $f_4=x_3x_5$ and $J=(x_1x_2x_4x_5)$. We have  $w_{12}=x_1x_2x_3$, $w_{13}=x_2x_3x_4$, $w_{14}=x_2x_3x_5$, $w_{34}=x_3x_4x_5$, $w_{23}=x_1x_2x_3x_4$, $w_{24}=x_1x_2x_3x_5$, $C_2=\{w_{23},w_{24}\}$,
$C=C_2\cup \{x_1w_{34},x_2x_3x_4x_5\}$, ${\tilde I}_1=\{x_1f_3,x_1f_4,f_2\}\supset J$, ${\tilde I}_4=\{f_3,x_4f_2\}\supset J$ and $B\cap (x_1)=\{x_1f_3,x_1f_4,x_4f_2,x_5f_2,w_{12}\}$, $B\cap (x_4)=\{w_{13},w_{14},w_{34},x_4f_2,x_1f_3\}$. Note that $\sdepth_S  U_1\leq d+1=3$,  $\sdepth_S  U_4\leq 3$ because $|B\cap (x_1)|=|B\cap (x_4)|=5>|C\cap (x_1)|+1=|C\cap (x_4)|+1=4$. Thus $\depth_S U_1=\depth_SU_4\leq 3$ and so we get $\depth_SI/J\leq 3$ using two different $t$.  }
\end{Example}

\begin{Theorem} \label{m1}  Suppose that $I \subset S$ is minimally generated by four  squarefree monomials $\{f_1,\ldots,f_4\}$ of degrees $d$ such that
there exists $c\in C$ such that $\supp c\not \subset \cup_{i\in [4]} \supp f_i$. If  $\sdepth_SI/J=d+1$ then $\depth_SI/J\leq d+1$.
\end{Theorem}
\begin{proof} Apply Theorem \ref{m0}, since Conjecture \ref{c} holds for $r<4$ by Theorem \ref{m}.
\end{proof}

\section{Proof of Theorem \ref{m}}

Suppose that $E\not =\emptyset$ and $s\leq q+r$.
For   $b=f_1x_{i}\in B$ set $I_b=(f_2,\ldots,f_r,B\setminus\{b\})$, $J_b=J\cap I_{b}$. If $\sdepth_SI_{b}/J_{b}\geq d+2$ then let $P_b$ be a partition on $I_{b}/J_{b}$ with sdepth $d+2$. We may choose $P_b$ such that each interval starting with a squarefree monomial of degree $d$, $d+1$ ends with a  monomial of $C$. In $P_{b}$ we  have some intervals $[f_k,f_kx_{i_k}x_{j_k}]$, $1<k\leq r$ and for all $b' \in B\setminus\{b,f_2x_{i_2},f_2x_{j_2},\ldots,f_rx_{i_r},f_rx_{j_r}\}]$ an interval $[b',c_{b'}]$. We define $h:[[\{f_2,\ldots,f_r\}\cup B]\setminus\{b,f_2x_{i_2},f_2x_{j_2},\ldots,f_rx_{i_r},f_rx_{j_r}\}]\to C$ by $f_k\to f_kx_{i_k}x_{j_k}$ and $b'\to c_{b'}$. Then $h$ is an injection and $|\Im\ h|= s-r\leq q$ (if $s=r+q$ then $h$ is a bijection).

\begin{Lemma}\label{ml1} Suppose that the following conditions hold:
\begin{enumerate}
\item{} $r=2$, $4\leq s\leq q+2$,
\item{} $C\subset ((f_1)\cap (f_2))\cup ((E)\cap (f_1,f_2))\cup (\cup_{a,a'\in E, a\not =a'}(a)\cap (a'))$,
\item{}  $\sdepth_SI_{b}/J_b\geq d+2$ for a $b\in (B\cap (f_1))\setminus (f_2)$.
\end{enumerate}
Then either  $\sdepth_SI/J\geq d+2$,  or there exists a nonzero   ideal $I'\subsetneq I$ generated  by a subset of $\{f_1,f_2\}\cup B$  such that  $\sdepth_S I'/J'\leq d+1$ for  $J'=J\cap I'$ and $\depth_SI/(J,I')\geq d+1$.
\end{Lemma}

\begin{proof}

 Since  $\sdepth_SI_{b}/J_b\geq d+2$ we consider  $h$ as above for a partition $P_b$  with sdepth $d+2$ of $I_b/J_b$. We have an interval $[f_2,c'_2]$ in $P_b$. Suppose that $B\cap [f_2,c'_2]=\{u,u'\}$.
 A sequence $a_1,\ldots, a_k$ is called a {\em path} from $a_1$ to $a_k$ if the following statements hold:

   (i) $a_l\in B\setminus \{b,u,u'\}$, $l\in [k]$,

   (ii) $a_l\not=a_j$ for $1\leq l<j\leq k$,

   (iii) $a_{l+1}|h(a_l)$ and $h(a_l)\not \in (b)$ for all $1\leq l<k$.

    This path is {\em weak}  if $h(a_j)\in (u,u')$ for some $j\in [k]$. It is {\em bad} if $h(a_k)\in (b)$ and it is {\em maximal} if either $h(a_k)\in (b)$, or all divisors from $B$ of $h(a_k)$ are in $\{u,u',a_1,\ldots,a_k\}$. If $a=a_1$ we say that the above path {\em starts with} $a$.

  By hypothesis $s\geq 4$ and so there exists
 $a_1\in B\setminus \{b,u,u'\}$. Set $c_1 = h( a_1)$.
  If $c_1\in (b)$ then the path $\{a_1\}$ is  maximal and bad.  We construct below, as an example,  a path with $k>1$. By recurrence choose if possible $a_{p+1}$  to be a divisor from $B$ of $c_p$, which is not in $\{b,u,u',a_1,\ldots,a_{p}\}$ and set $c_p=h(a_p)$, $ p \geq 1$. This construction ends at step $p=e$ if all divisors from $B$ of $c_{e}$ are in $\{b,u,u',a_1,\ldots,a_{e}\}$. If $c_i\not\in (b)$ for $1\leq i<e$ then $\{a_1,\ldots,a_{e}\}$ is a maximal path. If one $c_p\in  (u,u')$ then the constructed path is weak. If $c_{e}\in (b)$ then this path is  bad.
   We have three cases:

1) there exist  no weak path and no  bad path starting with $a_1$,

2) there exists a  weak path starting with $a_1$ but no  bad path starts with $a_1$,

3) there exists a  bad path starting with $a_1$.

 In {\bf the first case}, set $T_1=\{b'\in B: \mbox{there\ \ exists\ \ a\ \ path}\ \ a_1,\ldots,a_k\ \  \mbox{with}\ \ a_k=b'\}$, $G_1=B\setminus T_1$,
and  $I'_1 = (f_1,G_1)$, $I'_2 = (f_2,G_1)$, $I'_{12} = (f_1,f_2,G_1)$, $I'' = (G_1)$,  $J'_1 = I'_1 \cap J$, $J'_2 = I'_2 \cap J$,  $J'_{12} = I'_{12} \cap J$, $J'' = I''\cap J$. Note that $I''\not =0$ because $b\in I''$ and all divisors from $B$ of a monomial $c\in U_1=h(T_1)$ belong to $T_1$.  Consider the following exact sequence
$$0\to I'_{12}/J'_{12}\to I/J\to I/(J,I'_{12})\to 0.$$
If $U_1\cap (f_1,f_2)=\emptyset$ then the last  term has depth $\geq d+1$ and sdepth $\geq d+2$ using the restriction of $P_b$ to $(T_1)\setminus (J,I'_{12})$ since $h(b') \notin I'_{12}$, for all $ b' \in T_1$. When the first term has sdepth $\geq d+2$ then by  \cite[Lemma 2.2]{R}  the middle term  has  sdepth  $\geq d+2$. Otherwise, the first term has sdepth $\leq d+1$ and we may take $I'=I'_{12}$.

If $U_1\cap (f_1)=\emptyset$, but $b_2\in T_1\cap (f_2)$,  then in the following exact sequence
$$0\to I'_1/J'_1\to I/J\to I/(J,I'_1)\to 0$$
the last term has sdepth $\geq d+2$   since $h(b') \notin I'_{1}$, for all $ b' \in T_1$ and we may substitute the interval $[b_2,h(b_2)]$  from the restriction of $P_b$ by $[f_2,h(b_2)]$, the second monomial from  $[f_2,h(b_2)]\cap B$ being also in $T_1$. As above we get either $\sdepth_S I/J \geq d+2$, or  $\sdepth_S I'_1/J'_1 \leq d+1$,  $\depth_S I/(J,I'_1) \geq d+1$. Similarly, we do when $U_1\cap (f_2)=\emptyset$ but  $U_1\cap (f_1)\not=\emptyset$.

Now, suppose that $b_{1}\in T_1\cap (f_1)$ and $b_{2}\in T_1\cap (f_2)$. We claim to choose $b_1\not=b_2$ and such  that one from $h(b_{1}),h(b_{2})$ is not in $(w_{12})$, let us say $h(b_{1})\not\in (w_{12})$. Indeed, if $w_{12} \not\in B$ and $h(b_1),h(b_2)\in (w_{12})$  then necessarily  $h(b_1)=h(b_2)$ and it follows $b_1=b_2=w_{12}$ which is false.  Suppose that $w_{12}\in B$ and $h(b_2)=x_jw_{12}$. Then choose $b_1=x_jf_1\in T_1$. If $h(b_1)\in (w_{12})$ then we get
$h(b_1)=h(b_2)$ and so $b_1=b_2=w_{12}$ which is impossible.

In the following exact sequence
$$0\to I''/J''\to I/J\to I/(J,I'')\to 0$$ the last term has sdepth $\geq d+2$
 since we may replace the intervals $[b_{1},h(b_{1})]$, $[b_{2},h(b_{2})]$  of the restriction of $P_b$ to $(T_1)\setminus (J,I'')$ with the disjoint intervals
$[f_1,h(b_{1})]$, $[f_2,h(b_{2})]$. Also the last term has depth $\geq d+1$ because in the exact sequence
 $$0\to (f_2)/(J,I'')\cap (f_2)\to I/(J,I'')\to I/(J,I'',f_2)\to 0$$
 the end terms have depth $\geq d+1$ since $h(b_{1})\not \in (f_2)$, otherwise $h(b_{1}) \in (w_{12})$, which is false.
 As above we get either $\sdepth_S I/J \geq d+2$, or  $\sdepth_S I''/J'' \leq d+1$,  $\depth_S I/(J,I'') \geq d+1$.

 In {\bf the second case}, let  $a_1,\ldots, a_{t_1}$ be a weak path and set $c_j=h(a_j)$ for $j\in [t_1]$. We may suppose that $c_{t_1}\in (u)$, otherwise take a shorter path. Denote $T_1$, $U_1$ as in the first case, which we keep it fix even we will change a little $h$. Suppose that $a_{t_1}\in (f_2)$. Then change in $P_b$ the intervals $[a_{t_1},c_{t_1}]$, $[f_2,c'_2]$ by $[f_2,c_{t_1}]$, $[u',c'_2]$. Thus the new $c'_2$ is among $\{c_1,\ldots,c_{t_1}\}\subset U_1$, though the old $c'_2\not\in U_1$. Also the new $u'$ is in  $T_1$. However, if the old $u'$ is not a divisor of a $c$ from  $ U_1$, then the proof goes as in the first case using  $T'_1=T_1\cup \{u\}$, $G'_1=B\setminus T'_1$ with $I'=I'_2$, or $I'=I'_{12}$. Otherwise, $T'_1$ should be completed because $u'$ is not now in $[f_2,c'_2]$ and we may consider some paths starting with $u'$.
 Note that there exists a path from $a_1$ to $u'$ since  $u'$ is a divisor of a monomial from $ U_1$. It follows that there exist no bad path starting with $u'$. Take ${\tilde T_1}=T'_1\cup \{b'\in B: \mbox{there\ \ exists\ \ a\ \ path\ \ from}\ \ u'\ \  \mbox{to}\ \ b'\}$ and the proof goes as above with $\tilde T_1$ instead $T'_1$, that is with  $I'$ generated by a subset of $ \{f_1,f_2\}\cup {\tilde G}_1$ for ${\tilde G}_1=B\setminus {\tilde T}_1$.

Now suppose that $a_{t_1}\not\in (f_2)$ but there exists $1\leq v<t_1$ such that $a_v\in (f_2)$ and $a_v|c_{t_1}$.
 Then we may replace in $P_b$ the intervals $[a_{p},c_{p}], v \leq p \leq t_1$ with the intervals $[a_v,c_{t_1}],[a_{p+1},c_p], v \leq p < t_1$. The old $c_{t_1}$ becomes the new $c_v$, that is we reduce to the case when $u$ divides $c_v$ and $a_v\in (f_2)$, subcase solved above.

Remains to study the subcase when there exist no $a_v\in (f_2)$, $1\leq v\leq t_1$ with $a_v|c_{t_1}$. Then there exists an $a_{t_1+1}\in B\cap (f_2)$, $a_{t_1+1}\not= u$ such that $a_{t_1+1}|c_{t_1}$. Clearly, $a_{t_1+1}\not=u'$ because otherwise $c'_2=c_{t_1}$. We have two subcases:

$1')$ there exists a  path $a_{t_1+1},\ldots,a_l$ such that $h(a_l)\in (a_{v'})$ for some $1\leq v'\leq t_1$,

$2')$  for any  path  $a_{t_1+1},\ldots,a_p$, any $h(a_j)$, $t_1<j\leq p$ does not belong to $(a_1,\ldots,a_{t_1})$.

In the first subcase, we replace in $P_b$ the intervals $[a_{j},c_{j}], v' \leq j \leq l$ with the intervals $[a_{v'},c_{l}],[a_{j+1},c_j], v' \leq j < l$. The new $h(a_{t_1+1})$ is the old $c_{t_1}$ and we may proceed as above. In the second case, we set $$T_2=  \{b'\in B: \mbox{there\ \ exists\ \ a\ \ path\ \ from}\ \ a_{t_1+1}\ \  \mbox{to}\ \ b'\}.$$ Note that any path starting from $a_{t_1+1}$ can be completed to a path from $a_1$ by adding  the monomials $a_1,\ldots,a_{t_1}$. Thus there exists no bad path starting with $a_{t_1+1}$, otherwise we can get one starting from $a_1$, which is false.

 If there exists no weak path starting with  $a_{t_1+1}$ then we proceed as in the first case with $T_2$ instead $T_1$. If there exists a weak path starting with $a_{t_1+1}$ then we proceed as above in case 2) with $T'_2$, or ${\tilde T}_2$ instead $T'_1$, or ${\tilde T}_1$, except in the  subcase $2')$ when we will define similarly a $T_3$ given by the paths starting with a certain $a_{t_2+1}$.  Note that the whole set $\{a_1,\ldots,a_{t_2}\}$ has different monomials. After
several such steps we must arrive in the case $p=t_m$ when $\{a_1,\ldots,a_{t_m}\}$ has different monomials and the subcase $2')$ does not appear. We end this case using $T_m$, or $T'_m$, or ${\tilde T}_m$ instead $T_1$, or $T'_1$, or ${\tilde T}_1$.

In {\bf the third case}, let $a_1,\ldots,a_{t_1}$ be a  bad path starting with $a_1$. Set $c_j=h(a_j)$, $j\in [t_1]$.
Then $c_{t_1}=bx_{l_1}$ and let us say $b=f_1x_i$.
 If $a_{t_1}\in (f_1)$ then changing in $P_b$ the interval $[a_{t_1},c_{t_1}]$ by $[f_1,c_{t_1}]$ we get a partition on $I/J$ with sdepth $d+2$. Thus we may assume that $a_{t_1}\not\in (f_1)$. If $f_1x_{l_1}\in \{a_1,\ldots,a_{t_1-1}\}$, let us say   $fx_{l_1}=a_v$, $1\leq v<t_1$ then we may replace in $P_b$ the intervals $[a_{p},c_{p}], v \leq p \leq t_1$ with the intervals $[a_v,c_{t_1}],[a_{p+1},c_p], v \leq p < t_1$. Now we see that we have in $P_b$ the interval $[f_1x_{l_1},f_1x_ix_{l_1}]$ and switching it with the interval $[f_1,f_1x_ix_{l_1}]$ we get a partition with sdepth $\geq d+2$ for $I/J$.

Thus we may assume that $f_1x_{l_1} \notin \{ a_1,...,a_{t_1} \}$. Now set $a_{t_1+1}=fx_{l_1}$. Let \\ $a_{t_1+1},\ldots,a_k$ be a path starting with  $a_{t_1+1}$ and set $c_j=h(a_j)$, $t_1<j\leq k$.  If $a_{p}=a_v$ for $v\leq t_1$, $p>t_1$ then  change in $P_b$ the intervals $[a_j,c_j], v \leq j \leq p$ with the intervals $[a_v,c_{p}],[a_{j+1},c_j], v \leq j< p$. We have in $P_b$ an interval $[f_1x_{l_1},f_1x_ix_{l_1}]$ and switching it to $[f_1,f_1x_ix_{l_1}]$ we get a partition  with sdepth $\geq d+2$ for $I/J$. Thus we may suppose that in fact  $a_{p}\not \in \{b,a_1,\ldots, a_{p-1}\}$ for any $p>t_1$ (with respect to any path starting with $a_{t_1+1}$).
 We have  three subcases:

$1'')$ there exist  no weak path and no  bad path starting with $a_{t_1+1}$,

$2'')$ there exists a  weak path starting with $a_{t_1+1}$ but no  bad path starts with $a_{t_1+1}$,

$3'')$ there exists a  bad path starting with $a_{t_1+1}$.

  Set $T_2=\{b'\in B: \mbox{there\  exists\ a\  path}\ \ a_{t_1+1},\ldots,a_{k}\ \ \mbox{with}\ \ a_k=b'\}$, $G_2=B\setminus T_2$, $U_2=h(T_2)$ in the first subcase, and  see that $I'$ generated by a subset of $\{f_1,f_2\}\cup G_2$ chosen as above works.

  In the second subcase, let  $a_{t_1+1},\ldots, a_k$ be a weak path and set $c_j=h(a_j)$ for $t_1<j\leq k$. We may suppose that $c_k\in (u)$.   Changing $P_b$ we may suppose that the new $c'_2$ is in $U_2$ as above.
    If the old $u'$ was not a divisor of a $c\in U_2$ then the proof goes as in the first case  with $T'_2=T_2\cup \{u\}$, $I'=I'_2$. Otherwise, $T'_2$ should be completed to a ${\tilde T_2}$ similar  to ${\tilde T_1}$. The proof goes as above with $\tilde T_2$ instead $T'_2$.

  In the third subcase,
let $a_{t_1+1},\ldots,a_{t_2}$ be a  bad path starting with $a_{t_1+1}$ and set $c_j=h(a_j)$ for $j>t_1$. We saw that the whole set $\{a_1,\ldots,a_{t_2}\}$ has different monomials. As above
$c_{t_2}=bx_{l_2}$ and  we may reduce to the case when $f_1x_{l_2}\not \in \{a_1,\ldots,a_{t_1}\}$. Set $a_{t_2+1}=f_1x_{l_2}$ and again we consider three subcases, which we treat as above. Anyway after
several such steps we must arrive in the case $p=t_m$ when $b|c_{t_m}$ and again a certain $f_1x_{l_m}$ is not among $\{a_1,\ldots,a_{t_m}\}$ and  there exist no bad path starting with $a_{t_m+1}=f_1x_{l_m}$. This follows since  we may reduce to the case when the set $\{a_1,\ldots,a_{t_m}\}$ has different monomials and so the procedures should stop for some $m$.
 Finally, using $$T_m=\{b'\in B: \mbox{there\  exists\ a\  path}\ \ a_{t_m+1},\ldots,a_{k}\ \ \mbox{with}\ \ a_k=b'\}$$ (resp. $T'_m$, or ${\tilde T}_m$) as $T_1$ (resp. $T'_1$, or ${\tilde T}_1$) above we are done.
 \end{proof}

\begin{Lemma}\label{ml2} Suppose that the following conditions hold:
\begin{enumerate}
\item{} $r=3$, $6\leq s\leq q+3$,
\item{} $C\subset (\cup_{i,j\in [3],i\not=j} (f_i)\cap (f_j))\cup ((E)\cap (f_1,f_2,f_3))\cup (\cup_{a,a'\in E, a\not =a'}(a)\cap (a'))$,
\item{}  There exists $b\in (B\cap (f_1))\setminus (f_2,f_3)$ such that $\sdepth_SI_{b}/J_b\geq d+2$.
\end{enumerate}
Then either  $\sdepth_SI/J\geq d+2$, or there exists a nonzero   ideal $I'\subsetneq I$ generated  by a subset of $\{f_1,f_2,f_3\}\cup B$  such that  $\sdepth_S I'/J'\leq d+1$ for  $J'=J\cap I'$ and $\depth_SI/(J,I')\geq d+1$.
\end{Lemma}

\begin{proof}
We follow the proof of Lemma \ref{ml1}.
 Since  $\sdepth_SI_{b}/J_b\geq d+2$ we consider  $h$ as above for a partition $P_b$  with sdepth $d+2$ of $I_b/J_b$. We have two intervals $[f_2,c'_2]$, $[f_3,c'_3]$ in $P_b$. Suppose that $B\cap [f_i,c'_i]=\{u_i,u'_i\}$, $1<i\leq 3$ .
 As in Lemma \ref{ml1} we define   a  path $a_1,\ldots, a_k$ from $a_1$ to $a_k$ and a bad path.
The above path is {\em weak}  if $h(a_j)\in (u_2,u'_2,u_3,u'_3)$ for some $j\in [k]$. It is {\em maximal} if either $h(a_k)\in (b)$, or all divisors from $B$ of $h(a_k)$ are in $\{b,u_2,u'_2,u_3,u'_3,a_1,\ldots,a_k\}$.

By hypothesis $s\geq 6$ and there exists
$a_1\in B\setminus \{b,u_2,u_2',u_3,u'_3\}$. Set $c_1 = h( a_1)$.   If $c_1\in (b)$ then the path $\{a_1\}$ is  maximal and bad.  We construct below a path with $k>1$. By recurrence choose if possible $a_{p+1}$  to be a divisor from $B$ of $c_p$, which is not in $\{b,u_2,u'_2,u_3,u_3',a_1,\ldots,a_{p}\}$ and set $c_p=h(a_p)$, $ p \geq 1$. This construction ends at step $p=e$ if all divisors from $B$ of $c_{e}$ are in $\{b,u_2,u'_2,u_3,u_3',a_1,\ldots,a_{e}\}$.
  If $c_j\not\in (b)$ for $1\leq j<e$ then $\{a_1,\ldots,a_{e}\}$ is a maximal path. If one $c_p\in  (u_2,u'_2,u_3,u'_3)$ then the constructed path is weak. If $c_{e}\in (b)$ then this path is  bad.

  We  may reduce to the situation when $P_b$ satisfies the following property:

 $ (*)$ For all $1\leq i<j\leq 3$ if   $w_{ij}\in B\setminus \{u_2,u'_2,u_3,u'_3\}$, $1\leq i<j\leq 3$ then $h(w_{ij}) \not\in (u_i,u'_i,u_j,u'_j)$ if $i>1$.

Indeed, suppose that  $w_{ij}\in B\setminus\{u_2,u'_2,u_3,u'_3\}$ and $h(w_{ij})\in (u_j)$. Then  $h(w_{ij})=x_lw_{ij}$ for some $l\not\in \supp w_{ij}$ and   we must have let us say $u_j=x_lf_j$. Changing in $P_b$ the intervals $[f_j,c'_j]$, $[w_{ij},h(w_{ij})]$ with $[f_j,h(w_{ij})]$, $[u'_j,c'_j]$ we see that we may assume $u_j=w_{ij}$. Suppose that $(*)$ holds.
   We have three cases:

1) there exist  no weak path and no  bad path starting with $a_1$,

2) there exists a  weak path starting with $a_1$ but no  bad path starts with $a_1$,

3) there exists a  bad path starting with $a_1$.

 In {\bf the first case}, set $T_1=\{b'\in B: \mbox{there\ \ exists\ \ a\ \ path}\ \ a_1,\ldots,a_t\ \  \mbox{with}\ \ a_t=b'\}$, $G_1=B\setminus T_1$,
 and  for $k=(k_1,\ldots,k_m)$, $1\leq k_1<\ldots k_m\leq 3$, $0\leq m\leq 3$ set $I'_k = (f_{k_1},\ldots,f_{k_m},G_1)$,  $J'_k = I'_k \cap J$, and  $I'_0 = (G_1)$,  $J'_0 = I'_0\cap J$ for $m=0$. Note that  all divisors from $B$ of a monomial $c\in U_1=h(T_1)$ belong to $T_1$, and  $I'_0\not =0$ because $b\in I'_0$.  Consider the following exact sequence
$$0\to I'_{k}/J'_{k}\to I/J\to I/(J,I'_{k})\to 0.$$
If $U_1\cap (f_1,f_2,f_3)=\emptyset$ then the last  term of the above exact sequence given for $k=(1,2,3)$ has depth $\geq d+1$ and sdepth $\geq d+2$ using the restriction of $P_b$ to $(T_1)\setminus (J,I'_k)$ since $h(b') \notin I'_{k}$ , for all $ b' \in T_1$. When the first term has sdepth $\geq d+2$ then by  \cite[Lemma 2.2]{R}  the middle term  has  sdepth  $\geq d+2$ which is enough.

If $U_1\cap (f_1,f_2)=\emptyset$, but there exists $b_3\in T_1\cap (f_3)$,  then set $k=(1,2)$. In the following exact sequence
$$0\to I'_k/J'_k\to I/J\to I/(J,I'_k)\to 0$$
the last term has sdepth $\geq d+2$   since $h(b') \notin I'_{k}$, for all $ b' \in T_1$ and we may substitute the interval $[b_3,h(b_3)]$  from the restriction of $P_b$ by $[f_3,h(b_3)]$, the second monomial from  $[f_3,h(b_3)]\cap B$ being also in $T_1$. As above we get either $\sdepth_S I/J \geq d+2$, or  $\sdepth_S I'_1/J'_1 \leq d+1$,  $\depth_S I/(J,I'_1) \geq d+1$.

Now, we omit other subcases considering  only the worst subcase $m=0$. Let  $b_{1}\in T_1\cap (f_1)$, $b_{2}\in T_1\cap (f_2)$ and $b_{3}\in T_1\cap (f_3)$. For $1\leq l<j\leq 3$ we claim that we may  choose $b_l\not=b_j$ and such  that one from $h(b_{l}),h(b_{j})$ is not in $(w_{lj})$. Indeed, if $w_{lj} \not\in B$ and $h(b_l),h(b_j)\in (w_{lj})$  then necessarily  $h(b_l)=h(b_j)$ and it follows $b_l=b_j=w_{lj}$, which is false.  Suppose that $w_{lj}\in B$ and $h(b_j)=x_pw_{lj}$. Then choose $b_l=x_pf_l\in T_1$. If $h(b_l)\in (w_{lj})$ then we get
$h(b_l)=h(b_j)$ and so $b_l=b_j=w_{lj}$, which is impossible.

Therefore, we may choose $b_j$ such that $ h(b_{1})\not\in (w_{12})$, $h(b_{2})\not\in (w_{23})$.
 Note that it is possible that $f_1|c$ for some $c\in h(T_1)$ even $b\not |c$ for any $c\in U_1$.
 If $h(b_{1})\in (w_{13})$ then we may also choose $h(b_{3})\not\in (w_{13})$. In the case when $h(b_{1})\not\in (w_{13})$, choose any $b_3 \in T_1\cap (f_3)$ different from the others $b_j$.  We conclude that the possible intervals $[f_j,h(b_{j})]$, $j\in [3]$ are disjoint.
 Next we  change the intervals $[b_{j},h(b_{j})]$, $j\in [3]$  from the restriction of $P_b$ to $(T_1)\setminus (J,I'_0)$ by $[f_j,h(b_{j})]$, the second monomial from  $[f_j,h(b_{j})]\cap B$ being also in $T_1$.
 We claim that $I/(J,I'_0)$  has depth $\geq d+1$. Indeed, in the following  exact sequence
 $$0\to (f_2)/(f_2)\cap (J,I'_0,f_3)\to I/(J,I'_0,f_3)\to I/(J,I'_0,f_2,f_3)\to 0$$
 the first term has depth $\geq d+1$ because $h(b_{2})\not\in (f_2)\cap (f_3)$. If $h(b_{1})\not\in (f_3)$ then  $h(b_{1})\not \in(f_2,f_3)\cap (f_1)$ and so the last term has depth $\geq d+1$. If   $h(b_{1})\in (w_{13})$ then we may find a $b'\in (B\cap (f_1))\setminus (f_3)$ dividing $h(b_1)$. It follows that $b'\in T_1$ and  $b'\not\in  (f_2,f_3)\cap (f_1)$, which implies  that  the last term has again depth $\geq d+1$. Thus   $\depth_S I/(J,I'_0,f_3)\geq d+1$ by the Depth Lemma. Our claim follows from
 the exact sequence
 $$0\to (f_3)/(f_3)\cap (J,I'_0)\to I/(J,I'_0)\to I/(J,I'_0,f_3)\to 0$$
because the first term has depth $\geq d+1$. Therefore, as
above we get either \\ $\sdepth_S I/J \geq d+2$, or  $\sdepth_S I'_0/J'_0 \leq d+1$,  $\depth_S I/(J,I'_0) \geq d+1$.

 In {\bf the second case}, let  $a_1,\ldots, a_{t_1}$ be a weak path and set $c_j=h(a_j)$ for $j\in [t_1]$. We may suppose that $c_{t_1}\in (u_2)$, otherwise take a shorter path.   Denote $T_1$, $U_1$ as in the first case. First {\bf consider the subcase when} $U_1\cap (f_3)=\emptyset$. Suppose that $a_{t_1}\in (f_2)$. Then change in $P_b$ the intervals $[a_{t_1},c_{t_1}]$, $[f_2,c'_2]$ by $[f_2,c_{t_1}]$, $[u'_2,c'_2]$. Thus the new $c'_2$ is among $\{c_1,\ldots,c_{t_1}\}\subset U_1$, though the old $c'_2\not\in U_1$. If the old $u'_2$ is not a divisor of any $c\in U_1$ then the proof goes as in the first case  with $T'_1=T_1\cup \{u_2\}$. If the old   $u'_2$ is a divisor of a monomial from $ U_1$ then $T'_1$ should be completed because the old $u'_2$ is not now in $[f_2,c'_2]$. Note that there exists a path from $a_1$ to $u'_2$ since  $u'_2$ is a divisor of a monomial from $ U_1$. It follows that there exist no bad path starting with $u'_2$. It is worth to mention that the old $c'_2$ is now in $U_1$ and we should consider all pathes starting with divisors of $c'_2$ from $B$. Take ${\tilde T_1}=T'_1\cup \{b'\in B: \mbox{there\ \ exists\ \ a\ \ path\ \ from}\ \ u'_2\ \  \mbox{to}\ \ b'\}$ and the proof goes as above with $\tilde T_1$ instead $T'_1$, that is with $I'$ generated by a subset of  $\{f_1,f_2,f_3\}\cup {\tilde G}_1$, where ${\tilde G}_1=B\setminus {\tilde T}_1$.

Now suppose that $a_{t_1}\not\in (f_2)$ but there exists $1\leq v<t_1$ such that $a_v\in (f_2)$ and $a_v|c_{t_1}$.
 Then we may replace in $P_b$ the intervals $[a_{p},c_{p}], v \leq p \leq t_1$ with the intervals $[a_v,c_{t_1}],[a_{p+1},c_p], v \leq p < t_1$. The old $c_{t_1}$ becomes the new $c_v$, that is we reduce to the case when $u_2$ divides $c_v$ and $a_v\in (f_2)$, subcase solved above.

Remains to study the subcase when there exist no $a_v\in (f_2)$, $1\leq v\leq t_1$ with $a_v|c_{t_1}$. Then there exists an $a_{t_1+1}\in B\cap (f_2)$, $a_{t_1+1}\not= u_2$ such that $a_{t_1+1}|c_{t_1}$. Clearly, $a_{t_1+1}\not=u'_2$ because otherwise $c'_2=c_{t_1}$. We have two subcases:

$1')$ there exists a  path $a_{t_1+1},\ldots,a_l$ such that $h(a_l)\in (a_{v'})$ for some $1\leq v'\leq t_1$,

$2')$  for any  path  $a_{t_1+1},\ldots,a_p$, any $h(a_j)$, $t_1<j\leq p$ does not belong to $(a_1,\ldots,a_{t_1})$.

In the first subcase, we replace in $P_b$ the intervals $[a_{j},c_{j}], v' \leq j \leq l$ with the intervals $[a_{v'},c_{l}],[a_{j+1},c_j], v' \leq j < l$. The new $h(a_{t_1+1})$ is the old $c_{t_1}$ and we may proceed as above. In  the second subcase we set $$T_2=  \{b'\in B: \mbox{there\ \ exists\ \ a\ \ path\ \ from}\ \ a_{t_1+1}\ \  \mbox{to}\ \ b'\}.$$
 Note that any path starting from $a_{t_1+1}$ can be completed to a path from $a_1$ by adding  the monomials $a_1,\ldots,a_{t_1}$. Thus there exists no bad path starting with $a_{t_1+1}$, otherwise we can get one starting from $a_1$, which is false.

 If there exists no weak path starting with  $a_{t_1+1}$ then we proceed as in the first case with $T_2$ instead $T_1$. If there exists a weak path starting with $a_{t_1+1}$ then we proceed as above in case 2) with $T'_2$, or ${\tilde T}_2$ instead $T'_1$, or ${\tilde T}_1$, except in the  subcase $2')$ when we will define similarly a $T_3$ given by the paths starting with a certain $a_{t_2+1}$.  Note that the whole set $\{a_1,\ldots,a_{t_2}\}$ has different monomials. After
several such steps we must arrive in the case $p=t_m$ when $\{a_1,\ldots,a_{t_m}\}$ has different monomials and the subcase $2')$ does not appear. We end this case using $T_m$, or $T'_m$, or ${\tilde T}_m$ instead $T_1$, $T'_1$, or ${\tilde T}_1$. We should mention that if there exists $b_1\in T_m$ (or in $T'_m$, ${\tilde T}_m$) such that $h(b_1)\in (f_1)$ then changing $P_b$ as in case 1) we may suppose that $h(b_1)\not \in (w_{12})$ and $b_1\in T_m\cap (f_1)$. Thus we may consider the interval $[f_1,h(b_1)]$ disjoint of $[f_2,c'_2]$.

Consider  {\bf the subcase when there exist} $b_j\in  T_1$, $j=2,3$ such that $h(b_2)\in (u_2)$ and $h(b_3)\in (f_3)$ {\bf but} $h(T_1)\cap (u_3,u'_3)=\emptyset$. As above we may suppose that after several procedures we changed $P_b$ such that  $b_j\in (f_j)$ and the new $c'_2$ is  the old $h(b_2)$.
If $h(b_2)\not\in C_3\cup C_2$ then we may suppose that $h(b_2)\not \in(w_{12})$. As in the first case we may change $b_3$ such that $h(b_3)\not \in (w_{23})$. Indeed, the only problem could be if the old $h(b_3)\in \{u_3,u'_3\}$, which is not the case. We have no obstruction to change as usual $b_1$ such that $h(b_1)\not\in (w_{13})$ and so note that the interval $[f_2,h(b_2)]$ (resp. $[f_3,h(b_3)]$, or $[f_1,h(b_1)]$) has at most $w_{23}$ (resp. $w_{13}$, or $w_{12}$) from $W$. Thus the intervals $[f_j,h(b_j)]$, $j\in [3]$ are disjoint.

If $h(b_2)\in C_3$ then either $b_2=w_{23}$, or $w_{12}$. But $b_2\not =w_{23}$ because otherwise $h(w_{23})\in (u_2)$ contradicting $(*)$. Similarly, $b_2\not = w_{12}$. If $h(b_2)=w_{12}$ (resp. $h(b_2)=w_{23}$) then $b_2\not= w_{23}$ (resp. $b_2\not = w_{12}$) because otherwise we get a contradiction with $(*)$. Thus
$w_{12}$ (resp. $w_{23}$) is the only monomial of $W$ which belongs to $[f_2,h(b_2)]$. Choosing $b_3$ such that $h(b_3)\not \in (w_{23})$ (resp. $h(b_3)\not\in (w_{12})$) and $b_1$ such that $h(b_1)\not \in (w_{12})$ (resp. $h(b_1)\not\in (w_{13})$) we get disjoint the corresponding intervals.

Now consider  {\bf the subcase when there exist} $b_j\in  T_1$, $j=2,3$ such that $h(b_2)\in (u_2)$ {\bf and} $h(b_3)\in (u_3)$. If $h(b_2)\not\in (f_3)$ and $h(b_3)\not \in (f_2)$ then as above we may assume that with a different $P_b$, if necessary,  we may  reduce to the subcase when $b_j\in (f_j)$, $j=2,3$.
In general
this is not  simple because $h(b_2)$ as in Example \ref{bad}  can have no divisors from $B\cap (f_2)$, which are not in $\{u_2,u'_2,u_3,u'_3\}$ and there exist no  other $c\in U_1$ multiple of $u_2$. In such situation we are force to remain on the old $c'_2$ taking $T'_1=T_1\cup\{u_2,u'_2\}$ and $U'_1=U_1\cup \{c'_2\}$. If there exists a bad path starting on a divisor from $B\setminus \{u_2,u'_2,u_3,u'_3\}$ of $c'_2$ then we go to case 3). Otherwise, we should consider also the pathes starting with the divisors  of $c'_2$ from $B\setminus \{u_2,u'_2,u_3,u'_3\}$ completing $T'_1$ to ${\tilde T}_1$.
Note that because of $2')$ we may speak now about $T_m$ instead $T_1$.

Changing in $P_b$ the intervals $[f_j,c'_j]$,  $[b_j,h(b_j)]$, $j=2,3$ with $[f_j,h(b_j)]$, $[u'_j,c'_j]$, $j=2,3$ we may assume the new $c'_2,c'_3$ are in  $U_m=h(T_m)$ for some $m$ and the proof goes as above. If let us say  $h(b_2)\in (f_3)$  then we must be carefully since it is possible that the new intervals $[f_j,c'_j]$ could be not disjoint. A nice subcase  is for example when $h(b_2)$ is a least common multiple of $u_2,u_3$, which we study below.

 {\bf If} $w_{23}\in B$ then we {\bf we may suppose}  $u_2=w_{23}$. Indeed, if $w_{23}\not\in \{u_2,u'_2,u_3,u'_3\}$ and $u_2=x_pf_2$ for some $p\not\in \supp w_{23}$ then $h(b_2)=x_pw_{23}$. Since $b_2\not =u_2$ and $b_2\in (f_2)$ it follows that $b_2=w_{23}$. But this contradicts the property $(*)$.
 Suppose that $a_{t_1}=b_2$.  Then note that $h(b_2)=c_{t_1}=x_pw_{23}$ for some $p$ and it follows that $a_{t_1}=b_2=x_pf_2$ since $a_{t_1}\not = w_{23}=u_2$. Changing in $P_b$ the intervals $[f_2,c'_2]$,  $[b_2,h(b_2)]$ with $[f_2,h(b_2)]$, $[u'_2,c'_2]$,  we may assume the new $c'_2$ is in  $U_m$. We claim that $w_{23}$ is the only monomial from $B\cap W$ which is in $[f_2,c'_2]$. Indeed, $w_{12}$ could be another monomial from $B\cap W$ which  is present in the new $[f_2,c'_2]$. This could be true only if $a_{t_1}=w_{12}$. Thus $h(w_{12})=c_{t_1}\in (u_2)$ which is not possible again by $(*)$. The same procedure we use to include a new $c'_3$ in $U_m$. Since $u_2=w_{23}$ cannot be among
$u_3,u'_3$ we see that only $w_{13}$ could be among them. Suppose that $u_3=w_{13}$. Clearly the new $[f_3,c'_3]$ cannot contain  $w_{23}$. Choose as in the first case $b_1\in (f_1)$ such that $h(b_1)\not \in (w_{13})$ and the new intervals $[f_j,c'_j]$, $j\in [3]$ are disjoint.
If $w_{13}\not\in \{u_3,u'_3\}$ then we might have only $b_3=w_{13}$ and we may repeat the argument.

A problem could appear when the new $[f_j,c'_j]$, $j=2,3$ contain  $w_{12}$, $w_{13}$ because then we may not find $b_1$ as before. Note that this problem could
appear only when $w_{12},w_{13}\in \{u_2,u'_2,u_3,u'_3\}$ because of $(*)$. We will change the new $c'_2$ such that will not belong to $(f_1)$.
 Changing $P_b$  we may suppose  that $b_j\in (f_j)$, $j=2,3$ (again this change is not so simple as we saw above).
 We have $h(b_2)=c_{t_1}=x_pw_{12}$ for some $p$ and it follows that $a_{t_1}=b_2=x_pf_2$ since $a_{t_1}\not = w_{12}=u_2$. Suppose that $t_1>1$. Thus $a_{t_1}|c_{t_1-1}$ and we see that $c_{t_1-1}$ is not in $(f_1)$ because otherwise we get  $c_{t_1-1}=x_pw_{12}=c_{t_1}$, which is false.
{\bf If} $a_{t_1-1}\in (f_2)$ then  changing in $P_b$ the intervals $[a_{t_1},c_{t_1}]$, $[a_{t_1-1},c_{t_1-1}]$, $[f_2,c'_2]$ by $[f_2,c_{t_1-1}]$, $[u_2,c_{t_1}]$, $[u'_2,c'_2]$ we see that the new $c'_2$ is not in $(f_1)$ and belongs to $U_m$.
{\bf If} $w_{12}\in C$ then {\bf we get} $h(b_2)=w_{12}$ and the above argument works again, $c_{t_1-1}$ being the new $c'_2$.

{\bf When}
$a_{t_1-1}\not\in (f_2)$ {\bf  but} $u'_2|c_{t_1-1}$  we reduce the problem to the subcase when the path $\{a_1,\ldots,a_{t_1-1}\}$ goes from $a_1$ to $u'_2$ and now $u'_2\not\in (f_1)$. As above we may change $P_b$ such that the new $b_2=a_{t_1-1}\in (f_2)$ and the new $c'_2$, that is the old $c_{t_1-1}$ is not in $(f_1)$.

{\bf If}
$a_{t_1-1}\not\in (f_2)$, $u'_2\not |c_{t_1-1}$ but
 {\bf there exists} ${\tilde a}\in B\cap (f_2)$  a divisor  of $c_{t_1-1}$ then ${\tilde a}\not =u_2$ because otherwise we get $c_{t_1-1}=c_{t_1}$. Now we repeat the first part of the case 2).  If ${\tilde a}=a_v$ for some $1\leq v<t_1-1$ then
 changing in $P_b$ the intervals $[a_p,c_p]$, $v\leq p<t_1$ by $[a_v,c_{t_1-1}]$, $[a_{p+1},c_p]$,
 $v\leq p<t_1-1$ we see that the new $c_v$ (resp. $c_{v+1}$)  is the old $c_{t_1-1}$ (resp. $c_{t_1}$)). Now changing the intervals $[a_v,c_v]$, $[a_{v+1},c_{v+1}]$, $[f_2,c'_2]$ by $[f_2,c_v]$, $[w_{12},c_{v+1}]$, $[u'_2,c'_2]$ we see that the new $c'_2\not \in (f_1)$ and belongs to $ U_m$.
  If ${\tilde a}\not \in\{a_1,\ldots,a_{t_1}\}$ then we are  in one of the  above subcases  $1')$, $2')$ solved already.

We may use this argument to change $c'_j$, $j=2,3$  such that it is not in $(f_1)$ anymore, but as long as $h(b_j)\not = c_1$, that is the corresponding $t_1>1$. However, we may have  $h(b_j) = c_1$ only for one $j>1$, because if for instance  $h(b_3)=c_1$ then $c_1\in C_2\cup C_3$. If $c_1\in C_3$ then we see that $w_{23}\in B$ and $a_1=w_{23}$. But this contradicts $(*)$ because $h(w_{23})\in (u_2)$. If $c'_2\in C_2$ then $c'_2=w_{23}$ and either $a_1\in (f_2)$, or $a_1\in (f_3)$, that is $a_1$ cannot be $b_2$ and $b_3$ in the same time. Thus at least  one of the new $c'_j$, $j=2,3$ could be taken $\not \in (f_1)$. If let us say only  $c'_3\in (f_1)$ then choose $b_1\in T_1\cap (f_1)$ such that $h(b_1)\not\in (w_{13})$ as before. The interval $[f_1,h(b_1)]$ is disjoint from the other new constructed intervals, which is enough as we saw in case 1).

In {\bf the third case}, let $a_1,\ldots,a_{t_1}$ be a  bad path starting with $a_1$. Set $c_j=h(a_j)$, $j\in [t_1]$.
Then $c_{t_1}=bx_{l_1}$ and let us say $b=f_1x_i$.
 If $a_{t_1}\in (f_1)$ then changing in $P_b$ the interval $[a_{t_1},c_{t_1}]$ by $[f_1,c_{t_1}]$ we get a partition on $I/J$ with sdepth $d+2$. Thus we may assume that $a_{t_1}\not\in (f_1)$. If $f_1x_{l_1}\in \{a_1,\ldots,a_{t_1-1}\}$, let us say   $fx_{l_1}=a_v$, $1\leq v<t_1$ then we may replace in $P_b$ the intervals $[a_{p},c_{p}], v \leq p \leq t_1$ with the intervals $[a_v,c_{t_1}],[a_{p+1},c_p], v \leq p < t_1$. Now we see that we have in $P_b$ the interval $[f_1x_{l_1},f_1x_ix_{l_1}]$ and switching it with the interval $[f_1,f_1x_ix_{l_1}]$ we get a partition with sdepth $\geq d+2$ for $I/J$.

Thus we may assume that $f_1x_{l_1} \notin \{ a_1,...,a_{t_1} \}$. Now set $a_{t_1+1}=fx_{l_1}$. Let \\ $a_{t_1+1},\ldots,a_k$ be a path starting with  $a_{t_1+1}$ and set $c_j=h(a_j)$, $t_1<j\leq k$.  If $a_{p}=a_v$ for $v\leq t_1$, $p>t_1$ then  change in $P_b$ the intervals $[a_j,c_j], v \leq j \leq p$ with the intervals $[a_v,c_{p}],[a_{j+1},c_j], v \leq j< p$. We have in $P_b$ an interval $[f_1x_{l_1},f_1x_ix_{l_1}]$ and switching it to $[f_1,f_1x_ix_{l_1}]$ we get a partition  with sdepth $\geq d+2$ for $I/J$. Thus we may suppose that in fact  $a_{p}\not \in \{b,a_1,\ldots, a_{p-1}\}$ for any $p>t_1$ (with respect to any path starting with $a_{t_1+1}$).
 We have  three subcases:

$1'')$ there exist  no weak path and no  bad path starting with $a_{t_1+1}$,

$2'')$ there exists a  weak path starting with $a_{t_1+1}$ but no  bad path starts with $a_{t_1+1}$,

$3'')$ there exists a  bad path starting with $a_{t_1+1}$.

  Set $T_2=\{b'\in B: \mbox{there\  exists\ a\  path}\ \ a_{t_1+1},\ldots,a_{k}\ \ \mbox{with}\ \ a_k=b'\}$. We treat the subcases $1'')$, $2'')$ as the cases 1), 2) and find $I'$ generated by a subset of $\{f_1,f_2,f_3\}\cup G_2$, or $\{f_1,f_2,f_3\}\cup G'_2$, or $ \{f_1,f_2,f_3\}\cup {\tilde G}_2$,  where $G_2$, $G'_2$, ${\tilde G}_2$, are obtained from $T_2$ and as above $T'_2$, or ${\tilde T}_2$.

 In the subcase $3'')$,
let $a_{t_1+1},\ldots,a_{t_2}$ be a  bad path starting with $a_{t_1+1}$ and set $c_j=h(a_j)$ for $j>t_1$. We saw that the whole set $\{a_1,\ldots,a_{t_2}\}$ has different monomials. As above
$c_{t_2}=bx_{l_2}$ and  we may reduce to the case when $f_1x_{l_2}\not \in \{a_1,\ldots,a_{t_1}\}$. Set $a_{t_2+1}=f_1x_{l_2}$ and again we consider three subcases, which we treat as above. Anyway after
several such steps we must arrive in the case $p=t_m$ when either we may proceed as in the subcases $1'')$, $2'')$, or   $b|c_{t_m}$ and again a certain $f_1x_{l_m}$ is not among $\{a_1,\ldots,a_{t_m}\}$ and taking $a_{t_m+1}=f_1x_{l_m}$ there exist no  bad path starting with $a_{t_m+1}$. This follows since  we may reduce to the subcase when the set $\{a_1,\ldots,a_{t_m}\}$ has different monomials and so the procedures should stop for some $m$.
 Finally, using $$T_m=\{b'\in B: \mbox{there\  exists\ a\  path}\ \ a_{t_m+1},\ldots,a_{k}\ \ \mbox{with}\ \ a_k=b'\}$$ (resp. $T'_m$, or ${\tilde T}_m$) as $T_1$ (resp. $T'_1$, or ${\tilde T}_1$) above we are done.
\end{proof}

{\bf Proof of Theorem \ref{m}.}\ \ \ By Theorems \ref{po}, \ref{sh} we may suppose that $2r\leq s\leq q+r$ and we  may assume  that    $E$ contains only monomials of degrees $d+1$ by \cite[Lemma 1.6]{PZ}.  Apply induction on $|E|$. If $E=\emptyset$ we may apply Theorem \ref{pz}.  Suppose that $|E|>0$ and  $B\cap (f_1,\ldots,f_r)\not =\emptyset$, $r=2,3$, otherwise we get  $\depth_S I/J \leq d+1$ using \cite[Lemma 1.5]{PZ} applied to any $f_i$. We  may choose $b\in B\cap (f_1,f_2,f_3)$ which is not in $W$  if $r=2,3$ and $|B\cap (f_1,\ldots,f_r)|> 3\geq |B\cap W|$. However,  $|B\cap (f_1,\ldots,f_r)|<2r$ gives  $\depth_S(f_1,\ldots,f_r)/J\cap (f_1,\ldots,f_r)\leq d+1$ by Theorem \ref{sh} and    it follows that $\depth_SI/J\leq d+1$ using the Depth Lemma applied to the exact sequence
 $$0\to (f_1,\ldots,f_r)/J\cap (f_1,\ldots,f_r)\to I/J\to (E)/(J,f_1,\ldots,f_r)\cap (E)\to 0.$$
  Thus if $r=2,3$ we may  suppose to find $b\in B\cap (f_1,\ldots,f_r) \setminus W$.  Renumbering $f_i$ we may suppose that  $b\in (f_1)\setminus (f_2,\ldots,f_3)$.

 Apply induction on $r\leq 3$. Using Theorem \ref{pz} and induction  hypothesis on $|E|$ and $r$ apply \cite[Lemma 1.1]{PZ2}. Thus we may suppose that $C\subset ((f_1)\cap (f_2))\cup ((E)\cap (f_1,f_2))\cup (\cup_{a,a'\in E, a\not =a'}(a)\cap (a'))$,
 if $r=2$, or $C\subset (\cup_{i,j\in [3],i\not=j} (f_i)\cap (f_j))\cup ((E)\cap (f_1,f_2,f_3))\cup (\cup_{a,a'\in E, a\not =a'}(a)\cap (a'))$ if $r=3$.

Set $I'_b=(f_2,\ldots,f_r,B\setminus \{b\})$, $J'_b=I'_b\cap J$.  Clearly $b\not \in I'_b$ and so in the following exact sequence
$$0\to I'_b/J'_b\to I/J\to I/(J,I'_b)\to 0 $$
the last term has  depth $\geq d+1$. If the first term has sdepth $\leq d+1$ then it has depth $\leq d+1$ by induction hypothesis on $r$, case $r=1$ being done in Theorem \ref{pz}. Thus we may suppose that $\sdepth_SI'_b/J'_b\geq d+2$ and we may apply Lemmas \ref{ml1}, \ref{ml2}.
   Then we get either $\sdepth_SI/J\geq d+2$ contradicting our assumption, or
 there exists a nonzero   ideal $I'\subsetneq I$ generated by a subset $G$ of $ B$, or by $G$ and a subset of $\{f_1,f_2,f_3\}$  such that  $\sdepth_S I'/J'\leq d+1$ for  $J'=J\cap I'$ and $\depth_SI/(J,I')\geq d+1$.
  In the last case we see that  $\depth_S I'/J'\leq d+1$ by induction hypothesis on $r$, $|E|$, or by Theorem \ref{pz} and so $\depth_SI/J\leq d+1$ by the Depth Lemma applied to the following exact sequence
$$0\to I'/J'\to I/J\to I/(J,I')\to 0.\ \ \ \ \ \ \ \Box $$

The following bad example it is useful to illustrate somehow our proof.

\begin{Example} \label{bad} {
Let $n=6$, $r=3$, $d=1$, $f_i=x_i$ for $i\in [3]$, $E=\{x_4x_5,x_5x_6\}$,  $I=(x_1,x_2,x_3,E)$ and $J=(x_2x_4,x_3x_4,x_1x_2x_6,x_1x_3x_6,x_1x_4x_6,x_1x_5x_6,x_2x_3x_6,x_2x_5x_6,$\\ $x_3x_5x_6)$.  Then $B=\{x_1x_2,x_1x_3,x_2x_3,x_1x_4,x_1x_5,x_2x_5,x_3x_5\}\cup E$ and $$C=\{x_1x_2x_3,x_1x_2x_5,x_2x_3x_5,x_1x_3x_5,x_1x_4x_5,x_4x_5x_6\}.$$ Take $b=x_1x_4$ and $I_b=(x_2,x_3,x_1x_2,x_1x_3,x_1x_5,E)$, $J_b=I_b\cap J$. There exists a partition $P_b$ with sdepth $3$ on $I_b/J_b$ given by the intervals $[x_2,x_1x_2x_3]$, $[x_3,x_1x_3x_5]$, $[x_1x_5,x_1x_2x_5]$, $[x_2x_5,x_2x_3x_5]$, $[x_4x_5,x_1x_4x_5]$, $[x_5x_6,x_4x_5x_6]$. We have $c'_2=x_1x_2x_3$, $c'_3=x_1x_3x_5$ and $u_2=x_2x_3$, $u'_2=x_1x_2$, $u_3=x_3x_5$, $u'_3=x_1x_3$. Clearly, $u_2=w_{23}$.

Take $a_1=x_1x_5$, $c_1=x_1x_2x_5$, $a_2=x_2x_5$, $c_2=x_2x_3x_5$. The path $\{a_1,a_2\}$ is maximal weak because the  divisors from $B$ of $c_2$ are $a_2,u_2,u_3$. Then $T_1=\{a_1,a_2\}$ and we change
in $P_b$ as in the proof the intervals $[x_2,c'_2]$, $[a_2,c_2]$ by $[x_2,c_2]$, $[u'_2,c'_2]$. Thus the new $c'_2$ is the old $c_2$.
Now note that this new $c'_2$ is a multiple of $u_3$ and it is the only monomial from $h(T_1)$, which is a such multiple. Thus we had to take $u_3$ in the new $T'_1$, and $u'_3$ as well and certainly $c'_3$ is added to $h(T_1)$. Clearly, all divisors from $B$ of $c'_3$ are in $T'_1=T_1\cup \{u_3,u'_3\}$. But the former $u'_2$ divides $c_1$ and so should be added  to $T'_1$.  Thus we have $I'=(b,E)$, $J'=J\cap I'$ and  $I/(J,I')$  has a partition of sdepth $3$ given by the intervals $[x_2,c_2]$, $[x_3,c'_3]$, $[x_1,x_1x_2x_5]$. If $\sdepth_SI'/J'\geq 3$ then we get $\sdepth_SI/J\geq 3$, which is false. Otherwise,   $\sdepth_SI'/J'\leq 2$ and we get $\depth_SI'/J'\leq 2$ by \cite[Theorem 4.3]{P} and so  $\depth_SI/J\leq 2$ using the Depth Lemma.}
\end{Example}

\end{document}